\documentclass[12]{amsart}
\usepackage{amsmath,amssymb,amsthm,color,enumerate,comment,centernot,enumitem,url,cite}
\usepackage{graphicx,relsize,bm}
\usepackage{mathtools}
\usepackage{array}

\makeatletter
\newcommand{\tpmod}[1]{{\@displayfalse\pmod{#1}}}
\makeatother

\newtheorem{thm}{Theorem}[section]
\newtheorem{lemma}[thm]{Lemma}

\newtheorem{cor}[thm]{Corollary}

\theoremstyle{remark}

\theoremstyle{definition}
    \newtheorem{defn}[thm]{Definition}

\theoremstyle{THM}

\newcommand{\G}{{\mathcal G}}

\newcommand{\RR}{{\mathcal R}}

\newcommand{\Mod}[1]{\ (\mathrm{mod}\enspace #1)}
\newcommand{\mmod}[1]{\ \mathrm{mod}\enspace #1}
\newcommand{\Z}{{\mathbb Z}}
\newcommand{\Q}{{\mathbb Q}}

\newcommand{\F}{{\mathbb F}}
\newcommand{\abs}[1]{\left|{#1}\right|}

\makeatletter
\@namedef{subjclassname@2020}{%
  \textup{2020} Mathematics Subject Classification}
\makeatother

\def\red#1 {\textcolor{red}{#1 }}
\def\blue#1 {\textcolor{blue}{#1 }}

\numberwithin{equation}{section}

\begin{document}

\title[Generalized Wieferich primes and monogenic trinomials]{Generalized Wieferich primes\\ and monogenic trinomials}

\author{Amy Falk}
\address{Department of Mathematics, Cedar Crest College, Allentown, Pennsylvania 18104, USA}
\email[Amy~Falk]{A.Falk341@cedarcrest.edu}

\author{Joshua Harrington}
\address{Department of Mathematics, Cedar Crest College, Allentown, Pennsylvania 18104, USA}
\email[Joshua~Harrington]{joshua.harrington@cedarcrest.edu\\ ORCID: 0000-0001-7433-3666}

\author{Lenny Jones}
\address{Professor Emeritus, Department of Mathematics, Shippensburg University, Shippensburg, Pennsylvania 17257, USA}
\email[Lenny~Jones]{doctorlennyjones@gmail.com\\ ORCID: 0000-0001-7661-4226}

\date{\today}

\begin{abstract}
Let $b\ge 2$ be an integer and let $p\ge 3$ be a prime. We say that $p$ is a {\em generalized Wieferich prime base $b$}, or more succinctly, a {\em base-$b$ Wieferich prime,} if $b^{p-1}\equiv 1 \pmod{p^2}$. When $b=2$, $p$ is also known simply as a Wieferich prime. Let $f(x)\in {\mathbb Z}[x]$ be a monic polynomial of degree $N\ge 2$. We say that $f(x)$ is monogenic if $f(x)$ is irreducible over ${\mathbb Q}$ and $\{1,\theta,\theta^2,\ldots,\theta^{N-1}\}$ is a basis for the ring of integers of ${\mathbb Q}(\theta)$, where $f(\theta)=0$. Recently, the third author proved that $x^{2p}+2x^p+2$ is monogenic if and only if $p$ is not a Wieferich prime. In this article, we generalize this result to % establish a generalization of this result for 
$x^{2n}+bx^n+b$ with certain restrictions on $b\ge 2$ and $n\ge 3$. %,  and $n\ge 3$ is not divisible by any base-$b$ Wieferich prime.
\end{abstract}


\maketitle
\section{Introduction}\label{Section:Intro}
For an integer $b\ge 2$ and a prime $p\ge 3$, we say that $p$ is a {\em generalized Wieferich prime base $b$}, or more succinctly, a {\em base-$b$ Wieferich prime}, if $b^{p-1}\equiv 1 \pmod{p^2}$  \cite{Conrad}. When $b=2$, the prime $p$ is also known simply as a {\em Wieferich prime}. Base-$b$ Wieferich primes are quite scarce. For example, the only known Wieferich primes are 1093 and 3511, while there are no known base-47 Wieferich primes \cite{Conrad,OEIS}. %1093 and 3511 are the only known Wieferich primes.

Let $f(x)\in {\mathbb Z}[x]$ be a monic polynomial of degree $N\ge 2$. We say that $f(x)$ is {\em monogenic} if $f(x)$ is irreducible over ${\mathbb Q}$ and $\{1,\theta,\theta^2,\ldots,\theta^{N-1}\}$ is a basis for the ring of integers of ${\mathbb Q}(\theta)$, where $f(\theta)=0$.  It is well known \cite{Cohen} that 
  \begin{equation} \label{Eq:Dis-Dis}
\Delta(f)=\left[\Z_K:\Z[\theta]\right]^2\Delta(K),
\end{equation} where $\Delta(f)$ and $\Delta(K)$ denote the discriminants over $\Q$, respectively, of $f(x)$ and the number field $K$. We refer to $\left[\Z_K:\Z[\theta]\right]$ as the {\em index} of $f(x)$, and we denote it index($f$). 
Thus, from \eqref{Eq:Dis-Dis}, $f(x)$ is monogenic if and only if the index($f$)=1.
%\begin{equation*}\label{Mono}
%f(x) \ \mbox{is monogenic if and only if the index($f$)=1}.
%\end{equation*}

 Recently, the third author \cite{JonesWie} proved that 
 \begin{equation}\label{JonesWie}
 x^{2p}+2x^p+2 \quad \mbox{is not monogenic if and only if} \quad 2^{p-1}\equiv 1 \pmod{p^2}.
  \end{equation}  
 Motivated by \eqref{JonesWie}, the authors pondered  whether \eqref{JonesWie} could be generalized to trinomials of the form $x^{2p}+bx^p+b$, where $b\ge 3$ is an integer and $p$ is an odd prime. %, where $b$ and $b-4$ are squarefree integers with $b\ge 3$, and $p$ is an odd prime that does not divide $b(b-4)$. 
 One might surmise that the generalization could be that $x^{2p}+bx^p+b$ is not monogenic if and only if $b^{p-1}\equiv 1 \pmod{p^2}$. While this intuition turns out to be correct for $b=3$, it fails to be true when $b\ge 5$, where the situation is more complicated. An analysis of the approach used in \cite{JonesWie} regarding the monogenicity of $x^{2p}+2x^p+2$  reveals that the setting for an appropriate generalization of \eqref{JonesWie} should require that both $b$ and $b-4$ be squarefree. Otherwise, when $x^{2p}+bx^p+b$ is irreducible over $\Q$, it is never monogenic, regardless of whether $p$ is, or is not, a base-$b$ Wieferich prime. Therefore, to avoid this muddled situation, we assume in the sequel that 
 \begin{equation}\label{Eq:Ass}
 \begin{gathered}
   \mbox{$b\ge 2$ is an integer such that $b$ and $b-4$ are squarefree,}\\
   \mbox{and we let $D:=b^2-4b$.}
   \end{gathered} 
 \end{equation}     
   In this article, we use methods that differ from \cite{JonesWie} to establish a generalization of \eqref{JonesWie} that holds for all $b \ge 2$, subject to \eqref{Eq:Ass}, and consequently, we provide a new proof of \eqref{JonesWie}.   
 
  We also use the following definitions and notation. 
  \begin{defn}\label{Def:Res}
Let $f(x)\in \Z[x]$ and $g(x)\in \Z[x]$ be monic with respective degrees of $m$ and $n$. Then, the {\em resultant} of $f(x)$ and $g(x)$ is %, denoted $R(f,g)$,
\[Res(f,g)=(-1)^{mn}\prod_{i=1}^mg(\lambda_i),\]
where $\lambda_1,\lambda_2,\ldots ,\lambda_m$ are the zeros of $f(x)$. 
\end{defn} 
\noindent It is easy to see from Definition \ref{Def:Res} that $f(x)$ and $g(x)$ have a zero in common if and only if $Res(f,g)=0$.
  \begin{defn}\label{Def:Lucas}
  Define the Lucas sequence of the second kind $(V_n)_{n\ge -1}$ by %s $(U_n)_{n\ge 0}$ and $(V_n)_{n\ge 0}$ by 
  \begin{gather*}
 V_{-1}:=-1, \quad V_0=2,\quad V_1=-b \quad \mbox{and}\\
  V_{n}=-bV_{n-1}-bV_{n-2} \quad \mbox{for $n\ge 2$.}
\end{gather*} Then the characteristic polynomial of $(V_n)_{n\ge -1}$ is 
\begin{equation*}\label{G}
G(x):=x^2+bx+b,
\end{equation*} with zeros $\alpha$ and $\beta$, so that $\alpha+\beta=-b$, $\alpha\beta=b$ and 
$V_n=\alpha^n+\beta^n$.
\end{defn}
\begin{defn}\label{Def:Gpb}
 Let $p\ge 3$ be a prime and define
  \[\G_{p,b}(x):=G(x^p)=x^{2p}+bx^p+b.\]
\end{defn}
\noindent Note, since $b\ge 2$ is squarefree, %from \eqref{Eq:Ass}
it follows that $\G_{p,b}(x)$ is Eisenstein with respect to every prime divisor of $b$, and thus $\G_{p,b}(x)$ is irreducible over $\Q$.
 \begin{defn}
   For a prime $p\ge 3$, we let $\delta$ denote the Legendre symbol $(\frac{D}{p})$. 
 \end{defn}  
 Our main results are as follows. 
 \begin{thm}\label{Thm:Main}  Let $b$ and $D$ be as described in \eqref{Eq:Ass}, and let $p\ge 3$ be a prime.  %such that $D\not\equiv 0 \pmod{p}$. %, and let $p\ge 3$ be a prime such that $D\not\equiv 0 \pmod{p}$. Then 
      \begin{enumerate}
      \item \label{I0:Main Thm} If $\delta=0$, then $\G_{p,b}(x)$ is monogenic.
     \item \label{I1:Main Thm} If $\delta=1$, then $\G_{p,b}(x)$ is not monogenic if and only if 
     \begin{equation*}\label{condition delta=1}
     (2V_p+b^p+b)^2-D(b^{p-1}-1)^2 \equiv 0 \pmod{p^3};
     \end{equation*}
     \item \label{I2:Main Thm} If $\delta=-1$, then $\G_{p,b}(x)$ is not monogenic if and only if
     \begin{equation*}\label{condition delta=-1} 
    V_p\equiv -b \pmod{p^2} \quad \mbox{ and } \quad b^{p-1}\equiv 1 \pmod{p^2}.
    \end{equation*}
    \end{enumerate}
   \end{thm}
   \begin{cor}\label{Cor1:Main}
     Let $b$ and $D$ be as described in \eqref{Eq:Ass}, and let $p\ge 3$ be a prime such that $D\not\equiv 0 \pmod{p}$. If 
    \[V_p\equiv -b \pmod{p^2} \quad \mbox{ and } \quad b^{p-1}\equiv 1 \pmod{p^2},\] then $\G_{p,b}(x)$ is not monogenic.  
   \end{cor}
  The following corollary indicates that Theorem \ref{Thm:Main} can be refined in the special cases of $b\in \{2,3\}$.
 %\begin{cor}\label{Cor2:Main} Let $b\in \{2,3,5\}$, and suppose that $p\ge 3$ is a prime such that $D\not\equiv 0 \pmod{p}$. 
 %\begin{enumerate}
 %\item \label{I1:Main Cor} If $b\in \{2,3\}$, then $\G_{p,b}(x)$ is not monogenic if and only if
 %\[b^{p-1}\equiv 1 \pmod{p^2};\]
 % \item \label{I2:Main Cor} If $b=5$, then $\G_{p,5}(x)$ is not monogenic if and only if 
 %  \[V_p\equiv -b \pmod{p^2} \quad \mbox{and} \quad 5^{p-1}\equiv 1 \pmod{p^2}.\] %V_{p-\delta+1}+V_{p-\delta-1}
 %  \end{enumerate} 
 %\end{cor} 
 \begin{cor}\label{Cor2:Main} Let $b\in \{2,3\}$, and let $p\ge 3$ be a prime such that $D\not\equiv 0 \pmod{p}$. %and suppose that $p$ is an odd prime not dividing $D$. %\ge 3$ is a prime such that %$D\not\equiv 0 \pmod{p}$. 
 The trinomial $\G_{p,b}(x)$ is not monogenic if and only if $b^{p-1}\equiv 1 \pmod{p^2}$. %$p$ is a base-$b$ Wieferich prime.
 %\[b^{p-1}\equiv 1 \pmod{p^2}.\]
 \end{cor} 

The next theorem gives an extension of Theorem \ref{Thm:Main}.
\begin{thm}\label{Thm:Main2}
   Let $b$ and $D$ be as described in \eqref{Eq:Ass}, and let $n\ge 3$ be an integer. Define
  \[\G_{n,b}(x):=x^{2n}+bx^n+b.\] Then $\G_{n,b}(x)$ is monogenic if and only if $\G_{p,b}(x)$ is monogenic for every prime divisor $p$ of $n$. % if and only if no prime divisor $p$ of $n$ with $\delta=1$ satisfies \eqref{condition delta=1} and no prime divisor $p$ of $n$ with $\delta=-1$ satisfies \eqref{condition delta=-1}.   
\end{thm} 

Highlighting the obvious, the case $p=2$ (and $n=2$) has been excluded from the previously stated results. One reason for this omission is that  determining whether 2 is a base-$b$ Wieferich prime is simply an immediate consequence of the definition. Indeed, 2 is a base-$b$ Wieferich prime if and only if $b\equiv 1 \pmod{4}$. Regardless of this triviality, and for the sake of completeness, we present the following theorem that addresses the relationship between the monogenicity of $\G_{2,b}(x)$ and whether 2 is a base-$b$ Wieferich prime. We point out that this theorem is not a new result. In fact, it is merely a corollary of \cite[Theorem 2.6]{JonesNYJM2026}, and so we do not supply a proof here.  

% Regardless of this triviality, it turns out that the relationship between the monogenicity of $\G_{2,b}(x)$ and when 2 is a base-$b$ Wieferich prime is a bit surprising, and seems to run somewhat contrary to the situation that occurs in \eqref{JonesWie}. For the sake of completeness, we present the following theorem that addresses the case $p=2$. 
\begin{thm}\label{Thm:Main3}
 Let $b$ and $D$ be as described in \eqref{Eq:Ass}. Then $\G_{2,b}(x)$ is not monogenic if and only if %$b\equiv 1 \pmod{4}$. 
 $2$ is a base-$b$ Wieferich prime. 
\end{thm}

  %\begin{rem}\label{Rem:WSS}
  %  A prime $p$ such that $F_{p-\left(\frac{5}{p}\right)}\equiv 0 \pmod{p^2}$ is called a {\em Wall-Sun-Sun prime}. There are %currently no known Wall-Sun-Sun primes.
  %\end{rem}
 \section{Preliminaries}\label{Section:Prelims} Using a theorem due to Swan \cite{Swan}, we have 
 \begin{equation}\label{Swan}
 \Delta(\G_{p,b})=p^{2p}b^{2p-1}(b-4)^p.
 \end{equation} A standard tool used to determine the monogenicity of a monic irreducible polynomial $f(x)$ is known as Dedekind's Index Theorem \cite{Cohen}, which can determine whether a prime divisor of $\Delta(f)$ divides the index($f$). However, for trinomials in particular, Jakhar, Khanduja and Sangwan \cite{JKS2} have given a more streamlined version of Dedekind's theorem. In light of \eqref{Eq:Ass}, Definition \ref{Def:Gpb} and \eqref{Swan}, it turns out that an application of the theorem of Jakhar et al. to our specific trinomials $\G_{p,b}(x)$ reveals that the monogenicity of $\G_{p,b}(x)$ is determined completely by whether the single prime $p$ divides the index($\G_{p,b}$). More precisely, this application produces the following very concise theorem.
\begin{thm}\label{Thm:JKS}
 The trinomial $\G_{p,b}(x)$ is not monogenic if and only if 
\begin{equation*}
G(x) \quad \mbox{and}\quad
H(x):=\dfrac{bx^p+b+\left(-bx-b\right)^{p}}{p}
\end{equation*}
are not coprime in $\F_{p}[x]$.
\end{thm} The following refinement of Theorem \ref{Thm:JKS} will be a useful tool.
 \begin{cor}\label{Cor:JKS}
 The trinomial $\G_{p,b}(x)$ is not monogenic if and only if
 \[\G_{p,b}(\theta)\equiv 0 \pmod{p^2} \quad \mbox{for some $\theta\in \{\alpha,\beta\}$}.\] 
 \end{cor}
 \begin{proof}
  By Theorem \ref{Thm:JKS}, $\G_{p,b}(x)$ is not monogenic if and only if $pH(\theta)\equiv 0 \pmod{p^2}$ for some $\theta\in \{\alpha,\beta\}$. Since $G(\theta)\equiv 0 \pmod{p}$, we have that $-b\theta-b=\theta^2+pk$ for some $k\in \Z$. Consequently,  
  \[pH(\theta)=b\theta^p+b+(-b\theta-b)^p\equiv b\theta^p+b+(\theta^2+pk)^p\equiv \G_{p,b}(\theta) \pmod{p^2}.\qedhere\]
 \end{proof}
 The next theorem, due to Kaur, Kumar and Remete \cite{KKR}, addresses the monogenicity of compositions, and plays a crucial role in the proof of Theorem \ref{Thm:Main2}.  
 \begin{thm}\label{Thm:KKR}
   Let $f(x)\in \Z[x]$ be monic, and let $k\ge 2$ be an integer such that $F(x):=f(x^k)$ is irreducible over $\Q$. Then $F(x)$ is monogenic if and only if all of the following conditions are true:
   \begin{enumerate}
     \item \label{I1:KKR} $f(x)$ is monogenic,
     \item \label{I2:KKR} $p$ does not divide the {\rm index}($F$) for every prime divisor $p$ of $k$, 
     \item \label{I3:KKR} $f(0)$ is squarefree.
   \end{enumerate}  
 \end{thm}
 
\begin{lemma}\label{Lem:Equal}
$\G_{p,b}(\alpha)\equiv \G_{p,b}(\beta)\Mod{p^2}$ if and only if $V_p+b\equiv 0\Mod{p^2}$.  
\end{lemma} 
\begin{proof}
Observe that 
  \[\G_{p,b}(\beta)-\G_{p,b}(\alpha)=(\beta^p-\alpha^p)(\alpha^p+\beta^p+b).\] Since $D\not \equiv 0 \pmod{p}$, it follows that  $\beta^p-\alpha^p\not \equiv 0 \pmod{p^2}$. Therefore, 
  \begin{equation*}
  \begin{gathered}
  \G_{p,b}(\alpha)\equiv \G_{p,b}(\beta)\pmod{p^2}\\
  \mbox{if and only if}\\  
  \alpha^p+\beta^p+b=V_p+b\equiv 0 \pmod{p^2}.\qedhere
  \end{gathered}
  \end{equation*}  
\end{proof}
 
 The next result gives us some information regarding $\G_{p,b}(\theta)$ modulo $p^2$, where $\theta\in \{\alpha,\beta\}$, according to when $D$ is, or is not, a nonzero square modulo $p$.
  %the non-monogenicity of $\G_{p,b}(x)$ according to when $D$ is, or is not, a nonzero square modulo $p$. %how the behavior of the non-monogenicity of $\G_{p,b}(x)$ differs in the two cases, according to when $D$ is, or is not, a nonzero square modulo $p$.  
 \begin{cor}\label{Cor:delta} Suppose that $\G_{p,b}(\theta)\equiv 0 \pmod{p^2}$ for some $\theta\in \{\alpha,\beta\}$, so that $\G_{p,b}(x)$ is not monogenic by Corollary \ref{Cor:JKS}.
 \begin{enumerate}
  \item \label{delta:I2} If $\delta=1$, then $\G_{p,b}(-\theta-b)\equiv 0 \pmod{p^2}$ if and only if $V_p+b\equiv 0\pmod{p^2}$.
 \item \label{delta:I1} If $\delta=-1$, then $\G_{p,b}(-\theta-b)\equiv 0 \pmod{p^2}$ and $V_p+b\equiv 0\pmod{p^2}$. 
     \end{enumerate} 
  \end{cor}
  \begin{proof}
  Item \eqref{delta:I2} follows immediately from Lemma \ref{Lem:Equal}. 
  For item \eqref{delta:I1}, if $\delta=-1$, then $G(x)$ is irreducible in $\F_p[x]$. Let $\RR=(\Z/p^2\Z)[x]/(\G(x))$, 
  and let $\sigma$ be the unique automorphism of $\RR$ that fixes pointwise the elements of $\Z/p^2\Z$ and maps $\theta$ to $-\theta-b$. Then  
  \[\G_{p,b}(-\theta-b)\equiv \sigma(\G_{p,b}(\theta))\equiv \sigma(0)\equiv 0 \pmod{p^2},\] and $V_p+b\equiv 0\pmod{p^2}$ by Lemma \ref{Lem:Equal}.
 \end{proof}

%The next result follows from Definition \ref{Def:Res}, Corollary \ref{Cor:JKS} and Proposition \ref{Prop:delta}.
\begin{lemma}\label{Lem:Res1}
The trinomial $\G_{p,b}(x)$ is not monogenic if and only if 
\[Res(G,\G_{p,b})\equiv \left\{\begin{array}{cl}
  0 \pmod{p^4} & \mbox{if $\delta=-1$}\\
  0 \pmod{p^3} & \mbox{if $\delta=1$.}
\end{array}\right.\]
\end{lemma}
\begin{proof}
  By Definition \ref{Def:Res}, we have that 
  \begin{equation}\label{Res1} 
   Res(G,\G_{p,b})=\G_{p,b}(\alpha)\G_{p,b}(\beta)=(\alpha^{2p}+b\alpha^p+b)(\beta^{2p}+b\beta^p+b).
    \end{equation}
    Note that each factor in \eqref{Res1} is divisible by $p$ since $G(\alpha)\equiv G(\beta)\equiv 0\pmod{p}$.  
     By Corollary \ref{Cor:JKS}, $\G_{p,b}(x)$ is not monogenic if and only if at least one factor of \eqref{Res1} is divisible by $p^2$. Corollary \ref{Cor:delta} tells us that when $\delta=-1$, then $\G_{p,b}(x)$ is not monogenic if and only if both factors in \eqref{Res1} are zero modulo $p^2$ in the ring $\RR$.  
\end{proof}
The next lemma provides a more useful form for $Res(G,\G_{p,b})$.
 \begin{lemma}\label{Lem:Res2}
 \[ Res(G,\G_{p,b})=\frac{b}{4}\left((2V_p+b^p+b)^2-D(b^{p-1}-1)^2\right).\]
  % \begin{align*}
  % Res(G,\G_{p,b})&=b(V_p+b)(V_p+b^p)+b^2(b^{p-1}-1)^2\\
  % &=\frac{b}{4}\left((2V_p+b^p+b)^2-D(b^{p-1}-1)^2\right).
  % \end{align*} 
 \end{lemma}
 \begin{proof}
 We use the definitions in Section \ref{Section:Intro} and \eqref{Res1} to get that $V_{2p}=V_p^2-2b^p$ and
 \begin{align*}
  Res(G,\G_{p,b})&=\G_{p,b}(\alpha)\G_{p,b}(\beta)\\
  &=(\alpha^{2p}+b\alpha^p+b)(\beta^{2p}+b\beta^p+b)\\&=(\alpha\beta)^{2p}+b(\alpha\beta)^p(\alpha^p+\beta^p)
  +b(\alpha^{2p}+\beta^{2p})\\
  &\qquad \qquad +b^2(\alpha \beta)^p+b^2(\alpha^p+\beta^p)+b^2\\
  &=b^{2p}+b^{p+1}V_p+bV_{2p}+b^{p+2}+b^2V_p+b^2\\
  %&=b\left(V_p^2+(b^p+b)V_p+b^{p+1}\right)+b^2\left(b^{2p-2}-2b^{p-1}+1\right)\\
  %&=b(V_p+b)(V_p+b^p)+b^2(b^{p-1}-1)^2\\
  &=b\left(V_p^2+(b^p+b)V_p+b^{2p-1}+b^{p+1}-2b^p+b\right)\\
  &=\frac{b}{4}\left((2V_p+b^p+b)^2-(b^2-4b)(b^{p-1}-1)^2\right).\qedhere
 \end{align*}   
 \end{proof}

\section{The Proof of Theorem \ref{Thm:Main}} %\label{Section: Proof of Theorem \ref{Thm:Main}}
\begin{proof}%[Proof of Theorem \ref{Thm:Main}]
Suppose first that $\delta=0$, so that $p$ divides $D$. Thus, either $p$ divides $b$ or $p$ divides $b-4$. The first possibility that $p$ divides $b$ makes no sense when considering base-$b$ Wieferich primes. Notwithstanding, we can still investigate the monogenicity of $\G_{p,b}(x)$ in this case. Since $G(x)\equiv x^2\pmod{p}$, we can appeal to Corollary \ref{Cor:JKS} and examine $\G_{p,b}(0)$ modulo $p^2$, which is clearly never zero since $b$ is squarefree. Hence, $\G_{p,b}(x)$ is always monogenic in this situation. The second possibility is when $p$ divides $b-4$, but not $b$, so that \[G(x)\equiv x^2+4x+4\equiv (x+2)^2\pmod{p}.\] Using Corollary \ref{Cor:JKS} in this situation, we must examine \[\G_{p,b}(-2)=(-2)^{2p}+b(-2)^p+b\] modulo $p^2$. We claim that $\G_{p,b}(-2)\not \equiv 0 \pmod{p^2}$. Assume, by way of contradiction, that $\G_{p,b}(-2)\equiv 0 \pmod{p^2}$. Then, a little algebra shows that 
\[(b-4)(2^p-1)\equiv (2^p-2)^2\pmod{p^2}.\] Since $2^p-2\equiv 0 \pmod{p}$ and $2^p-1\not \equiv 0 \pmod{p}$, it follows that $b-4\equiv 0 \pmod{p^2}$, which contradicts the fact that $b-4$ is squarefree. Thus, the claim is established, and we conclude that $\G_{p,b}(x)$ is always monogenic, which proves item \eqref{I0:Main Thm}. 

For items \eqref{I1:Main Thm} and \eqref{I2:Main Thm}, let 
\begin{equation}\label{A and B}
A:=2V_p+b^p+b\quad \mbox{and} \quad B:=b^{p-1}-1.
\end{equation} 
Let $\theta\in \{\alpha,\beta\}$. Then  
\[\G_{p,b}(\theta)=G(\theta^p)\equiv G(\theta)^p\equiv 0 \pmod{p},\]  and since $\alpha\not \equiv \beta \pmod{p}$, we get that  
\begin{align*}
 A&\equiv 2(V_p+b) \pmod{p}\\
 &\equiv 2(\alpha^p+\beta^p+b) \pmod{p}\\ 
 &\equiv \frac{2(\alpha^p-\beta^p)(\alpha^p+\beta^p+b)}{\alpha^p-\beta^p} \pmod{p}\\
 &\equiv \frac{2\left(\G_{p,b}(\alpha)-\G_{p,b}(\beta)\right)}{\alpha^p-\beta^p} \pmod{p}\\
 &\equiv 0 \pmod{p}.
\end{align*}
Note also that $B\equiv  0\pmod{p}$ by Fermat's Little theorem.

Then, from Corollary \ref{Cor:delta},  Lemma \ref{Lem:Res1} and Lemma \ref{Lem:Res2}, we have that \begin{equation}\label{Refine1}
\begin{split}
\begin{gathered}
\mbox{$\G_{p,b}(x)$ is not monogenic}\\ 
\mbox{if and only if}\\
  A^2-DB^2=\left(A-B\sqrt{D}\right)\left(A+B\sqrt{D}\right)\equiv \left\{\begin{array}{cl}
  0 \pmod{p^3} & \mbox{if $\delta=1$}\\
  0 \pmod{p^4} & \mbox{if $\delta=-1$,}
\end{array}\right.
\end{gathered}
\end{split}
\end{equation}  where 
\[A-B\sqrt{D} \equiv 0 \Mod{p^2} \quad \mbox{or} \quad A+B\sqrt{D} \equiv 0 \Mod{p^2} \quad \mbox{if $\delta=1$}\]
and 
\[A-B\sqrt{D} \equiv 0 \Mod{p^2} \quad \mbox{and} \quad A+B\sqrt{D} \equiv 0 \Mod{p^2} \quad \mbox{if $\delta=-1$}.\]  Furthermore, if $\delta=-1$, then 
\begin{equation}\label{Refine2}
\begin{split}
\begin{gathered}
\mbox{$\G_{p,b}(x)$ is not monogenic}\\
\mbox{if and only if}\\
  A\equiv 0 \pmod{p^2} \quad \mbox{and} \quad B\equiv 0 \pmod{p^2}.
\end{gathered}
\end{split} 
\end{equation} Thus, if $B=b^{p-1}-1\equiv 0 \pmod{p^2}$, it follows that  
\[2(V_p+b)=2V_p+b+b\equiv 2V_p+b^p+b=A \pmod{p^2},\]
which completes the proof of the theorem. 
\end{proof}
%Therefore, we have shown that \eqref{Refine1} and \eqref{Refine2} have established, respectively, items \eqref{I3b:Main} and \eqref{I3a:Main} of the theorem. 
\section{The Proof of Corollary \ref{Cor1:Main}}
\begin{proof}
  If $\delta=-1$, then the corollary follows immediately from item \eqref{I2:Main Thm} of Theorem \ref{Thm:Main}. So, suppose that  $\delta=1$. Observe that 
  \[2V_p+b^p+b\equiv 2(-b)+b+b\equiv 0 \pmod{p^2},\] since $V_p\equiv -b\pmod{p^2}$ and $b^{p-1}\equiv 1\pmod{p^2}$. Hence, the corollary then follows from item \eqref{I1:Main Thm} of Theorem \ref{Thm:Main}. 
\end{proof}
\section{The Proof of Corollary \ref{Cor2:Main}}
\begin{proof} 
To establish this corollary,  
we need to conduct a more-detailed analysis of \eqref{Refine1} and \eqref{Refine2} when $b\in \{2,3\}$. For integers $r$ and $m$ with $m\ge 2$, we let the notation ``$r \mmod{m}$" denote the unique integer $z\in \{0,1,2,\ldots,m-1\}$ such that $r\equiv z \pmod{m}$. That is, $r \mmod{m}=z$. 

Straightforward induction arguments show that
\begin{equation}\label{b=2 V_p}
 V_n=(-1)^{\frac{n^2+7}{8}}2^{\frac{n+1}{2}}, 
\end{equation} when $b=2$ and $n\equiv 1\pmod{2}$, and that 
\begin{equation}\label{b=3 V_p}
V_n=\left\{\begin{array}{rl}
  -3^{\frac{n+1}{2}} & \mbox{if $n \mmod{12}\in \{1,11\}$}\\[.5em]
  3^{\frac{n+1}{2}} & \mbox{if $n \mmod{12}\in \{5,7\}$},
\end{array} \right. 
\end{equation} when $b=3$. The unique structure of $\abs{V_n}$ being a power of $b$, as given in \eqref{b=2 V_p} and \eqref{b=3 V_p}, is the driving force behind the refinement of Theorem \ref{Thm:Main} when $b\in \{2,3\}$. One way to see that this phenomenon does not occur when $b\ge 5$ is to examine $V_7=-b^4v(b)$, where $v(b)=b^3-7b^2+14b-7$. It is easy to see that the equation $v(b)=-1$ has only the integer solution $b=3$, while the equation $v(b)=1$ has exactly the three solutions $b\in \{1,2,4\}$. Using the Rational Zero theorem, we confirm, when $k\ge 1$, that the equation $v(b)=-b^k$ has no integer solutions, and that the equation $v(b)=b^k$ has only the integer solution $b=1$. 

 Since all possibilities for $(p,b)$, where $b\in \{2,3\}$ and $p$ is a prime not dividing $D$, can be addressed in a similar manner, we provide details only for the case $b=3$.
Suppose first that $B=3^{p-1}-1\equiv 0\pmod{p^2}$. Then  
  \begin{equation}\label{3 cong}
 3^{\frac{p+1}{2}}\equiv \left\{\begin{array}{cl}
 3 \pmod{p^2} & \mbox{if $p \mmod{12}\in \{1,11\}$}\\[.5em]
 -3 \pmod{p^2} & \mbox{if $p \mmod{12}\in \{5,7\}$},
\end{array}\right.
\end{equation}
  by Euler's criterion and quadratic reciprocity. Thus,  
in any case of $p\mmod{12}$, we see from \eqref{A and B}, \eqref{b=3 V_p} and \eqref{3 cong} that $V_p\equiv -3\pmod{p^2}$, which implies that 
\[A\equiv 2(-3)+3+3 \equiv 0 \pmod{p^2}.\]
%which implies that $V_p\equiv -b\pmod{p^2}$. 
Hence, $\G_{p,3}(x)$ is not monogenic from Theorem \ref{Thm:Main}.

Conversely, suppose that $\G_{p,3}(x)$ is not monogenic. If $p\mmod{12}\in \{5,11\}$, then $\delta=-1$ since $D=-3$, and we have that $B\equiv 0 \pmod{p^2}$ from item \eqref{I2:Main Thm} of Theorem \ref{Thm:Main}. So, suppose $p\mmod{12}\in \{1,7\}$ so that $\delta=1$. 
By item \eqref{I1:Main Thm} of Theorem \ref{Thm:Main}, it follows that   
\[(A-B\sqrt{D})(A+B\sqrt{D})\equiv 0 \pmod{p^3},\] where 
\[A-B\sqrt{D} \equiv 0 \Mod{p^2} \quad \mbox{or} \quad A+B\sqrt{D} \equiv 0 \Mod{p^2}.\] Suppose that 
\begin{equation}\label{B cong}
  A-B\sqrt{D} \equiv 0 \pmod{p^2}.
\end{equation} 
Solving the congruence \eqref{B cong} for $B$ yields
\[B\equiv \frac{A}{\sqrt{D}}\equiv \frac{2V_p+3^p+3}{\sqrt{D}} \pmod{p^2}.\] Thus,
\[B\equiv \left\{\begin{array}{cl}
 \dfrac{2\left(-3^{\frac{p+1}{2}}\right)+3^p+3}{\sqrt{D}}\equiv \dfrac{3\left(3^{\frac{p-1}{2}}-1\right)^2}{\sqrt{D}}\equiv 0  \Mod{p^2} & \ \mbox{if $p \mmod{12}=1$}\\[1em] 
  \dfrac{2\left(3^{\frac{p+1}{2}}\right)+3^p+3}{\sqrt{D}} \equiv \dfrac{3\left(3^{\frac{p-1}{2}}+1\right)^2}{\sqrt{D}}\equiv 0  \Mod{p^2} & \ \mbox{if $p \mmod{12}=7$.}
\end{array}\right.\]
 If $A+B\sqrt{D} \equiv 0 \pmod{p^2}$, we arrive at the same conclusion.
 \end{proof}  

\section{The Proof of Theorem \ref{Thm:Main2}}
\begin{proof} Since $b$ and $b-4$ are squarefree, it follows that $D/\gcd(2,b)^2$ is squarefree and that $b\mmod{4}=2$ if $b$ is even. Thus, $G(x)$ is monogenic by \cite[Theorem 2.2]{JonesNYJM2026}. Since $G(0)=b$ is squarefree and $\G_{n,b}(x)=G(x^n)$, we conclude from Theorem \ref{Thm:KKR} that $\G_{n,b}(x)$ is monogenic if and only if $G(x^p)=\G_{p,b}(x)$ is monogenic for every prime divisor $p$ of $n$.  
\end{proof}

\section{Examples}

With $A$ and $B$ as defined in \eqref{A and B}, and $V_p$ as defined in Definition \eqref{Def:Lucas}, Table \ref{T:1} illustrates Theorem \ref{Thm:Main} for pairs $(p,b)$ by giving the values for %, \ Res(G,G_{p,b})
\[\delta, \quad (\overline{A}, \ \overline{B}, \ \overline{V_p+b}) \quad \mbox{and} \quad  (A^2-DB^2) \mmod{p^{N}},\] 
% \[(\G_{p,b}(\alpha), \  \G_{p,b}(\beta),\  V_p+b, \ b^{p-1}-1) \ \mmod{p^2},\] %when $\delta=1$, 
 along with the monogenicity of $\G_{p,b}(x)$, where $\overline{*}$ is $* \mmod{p^2}$ and $N=(7-\delta)/2$. 
 %Observe that 
 %\[N=\left\{\begin{array}{cl}
 %  3 & \mbox{if $\delta=1$}\\
 %  4 & \mbox{if $\delta=-1$.}
 %\end{array}\right.\] % as predicted by Corollary \eqref{Cor:JKS}. 
 
 \begin{table}[h]
 \begin{center}
\begin{tabular}{ccccc}
$(p,b)$ & $\delta$ & $(\overline{A}, \ \overline{B}, \ \overline{V_p+b})$  & $(A^2-DB^2) \mmod{p^{N}}$ &  $\G_{p,b}(x)$ is monogenic  \\ \hline
$(5,7)$ & 1 & $(5, \ 0, \ 15)$ & 25 & Yes\\
%$(5,20771)$ & 1 & $(346564135, \ 0, \ 388999288)$ & 7025047002803 & Yes\\
$(11,6)$ & 1 & $(88, \ 55, \ 0)$ & 726  & Yes\\
$(13,170)$ & 1 & $(0, \  0, \ 0)$ & 0 & No\\
$(37,7)$ & 1 & $(629, \ 1110, \ 1221)$ & 0 & No\\
$(71,11)$ & 1 & $(3266, \ 0, \ 1633)$ & 287337 & Yes\\
$(5,3)$ & $-1$ & $(0, \ 5, \ 5)$ & 450 & Yes\\
$(5,11)$ & $-1$ & $(10, \ 15, \ 10)$ & 25 & Yes\\
$(5,23)$ & $-1$ & $(20, \ 15, \ 0)$ & 450 & Yes\\
$(5,26)$ & $-1$ & $(0, \ 0, \ 0)$ & 0 & No\\
$(5,93)$ & $-1$ & $(15, \ 0, \ 20)$ & 600 & Yes\\
$(11,3)$ & $-1$ & $(0, \ 0, \ 0)$ & 0 & No\\
$(31,115)$ & $-1$ & $(0, \ 0, \ 0)$ & 0 & No
\end{tabular}
\end{center}
\caption{Data for pairs $(p,b)$ and the monogenicity of $\G_{p,b}(x)$}
 \label{T:1}
\end{table}

\end{document}